\documentclass[11pt]{article}

\usepackage[latin1]{inputenc}
\usepackage{amssymb}
\usepackage{latexsym}
\usepackage{xspace}
\usepackage{graphicx,color}
\usepackage{amsmath}
\usepackage{amsthm}
\usepackage{amsfonts}
\usepackage{float}
\usepackage{times}
\usepackage{alg}
\usepackage{anysize}
\usepackage{enumerate}
\usepackage{hyperref}

\newtheorem{theorem}{Theorem}

\newtheorem{lemma}{Lemma}

\newtheorem{corollary}{Corollary}
\newtheorem{conjecture}{Conjecture}

\newtheorem{observation}{Observation}

\newtheorem{claim}{Claim}

\newcommand{\paths}{\mathcal{P}}

\newcommand{\setZ}{\mathbb{Z}}

\newcommand{\setR}{\mathbb{R}}

\newcommand{\vpn}{{\tt sVPND}\xspace}
\newcommand{\crp}{{\tt PR}\xspace}
\newcommand{\mr}{{\tt MR}\xspace}

\begin{document}

\title{A Short Proof of the VPN Tree Routing Conjecture on Ring Networks}

\author{Fabrizio Grandoni\thanks{Dipartimento di Informatica, Sistemi e Produzione, Universit\`{a} di Roma ``Tor Vergata'', Italy, 
\texttt{grandoni@disp.uniroma2.it}} \and Volker Kaibel\thanks{Fakult\"{a}t f\"{u}r Mathematik, Otto-von-Guericke-Universit\"{a}t Magdeburg, Germany, \texttt{kaibel@ovgu.de}} \and Gianpaolo Oriolo\thanks{Dipartimento di Ingegneria dell'Impresa, Universit\`{a} di Roma ``Tor Vergata'', Italy, 
\texttt{oriolo@disp.uniroma2.it}} \footnote{Corresponding author.} \and Martin Skutella\thanks{Institut f\"{u}r Mathematik, Technische Universit\"{a}t Berlin, Germany, \texttt{skutella@math.tu-berlin.de}}}
\date{\today}

\maketitle

\begin{abstract}
Only recently, Hurkens, Keijsper, and Stougie proved the VPN Tree Routing Conjecture for the special case of ring networks.  We present a short proof of a slightly stronger result which might also turn out to be useful for proving the VPN Tree Routing Conjecture for general networks.
\end{abstract}


\section{Introduction}

Consider a communication network which is represented by an undirected graph $G=(V,E)$ 
with edge costs $c:E\rightarrow\setR_{\geq0}$.  Within this network there is a set of $k$ terminals $W\subseteq V$ which want to communicate with each other.  However, the exact amount of traffic between pairs of terminals is not known in advance.  Instead, each terminal $i\in W$ has an upper bound $b(i)\in\setZ_+$ on the cumulative amount of traffic that terminal $i$ can send or receive.  The general aim is to install capacities on the edges of the graph supporting any possible communication scenario at minimum cost where the cost for installing one unit of capacity on edge $e$ is $c(e)$.

A \emph{set of traffic demands} $D=\{d_{ij}\mid i,j\in W\}$ specifies for each unordered pair of terminals $i,j\in W$ the amount $d_{ij}\in\setR_{\geq0}$ of traffic between $i$ and $j$.  A set $D$ is \emph{valid} if it respects the upper bounds on the traffic of the terminals. That is, (setting $d_{ii}=0$ for all $i\in W$)
\begin{displaymath}
  \sum_{j \in W}d_{ij} \leq b(i)\qquad\text{for all terminals $i\in W$.}
\end{displaymath}

A solution to the \emph{symmetric Virtual Private Network Design} (\vpn) problem defined by $G$, $c$, $W$, and $b$ consists of an $i$-$j$-path $P_{ij}$ in $G$ for each unordered pair $i,j\in W$, and edge capacities $u(e)\geq 0$, $e\in E$.
Such a set of paths $P_{ij}$, $i,j\in W$, together with edge capacities $u(e)$, $e\in E$, is called a \emph{virtual private network}.  A virtual private network is \emph{feasible} if all valid sets of traffic demands $D$ can be routed without exceeding the installed capacities $u$ where all traffic between terminals $i$ and $j$ is routed along path $P_{ij}$, that is, (with $\binom{W}{2}$ denoting the set of cardinality-two subsets of~$W$)
\begin{align*}
  u(e)\geq\sum_{\{i,j\}\in\binom{W}{2}: e\in P_{ij}}d_{ij}\qquad\text{for all edges $e\in E$.}
\end{align*}
A feasible virtual private network is called \emph{optimal} if the total cost of the capacity reservation $\sum_{e\in E}c(e)\,u(e)$ is minimal.

A well-known open question is whether the \vpn problem can be solved efficiently
(i.e., in polynomial time); see Erlebach and R\"{u}egg \cite{ERL} and Italiano, Leonardi, and Oriolo \cite{ILO02}.  
A feasible virtual private network is a \emph{tree solution}
if the subgraph of $G$ induced by the support of $u$ (i.e., edges $e \in E$ with
$u(e)>0$) is a tree.  Gupta, Kleinberg, Kumar, Rastogi, and Yener \cite{GuptaKleinbergKumarRastogiYener01} prove that a tree solution of minimum cost can be obtained in polynomial time by an all-pair shortest paths computation on the network $G$.  In Figure~\ref{fig:tree-reservation} we explain how optimal (i.e., minimal) capacities can be determined for a fixed sub-tree of~$G$ spanning all terminals.
\begin{figure}[tb]
	\centering
		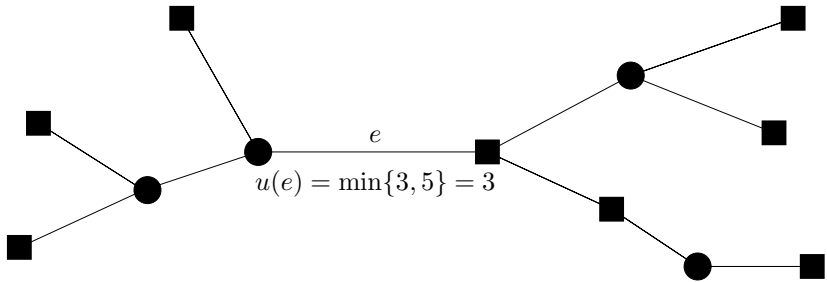
	\caption{A tree solution for an \vpn instance. The $8$ rectangular nodes are the terminals in $W$, the remaining nodes are non-terminal nodes.  For the sake of simplicity we assume in this instance that $b(j)=1$ for all terminals $j\in W$.  If some edge $e$ is removed from the tree, two connected components remain. The smaller component on the left hand side contains $3$ terminals while the larger component on the right hand side contains $5$ terminals.  Therefore the maximum amount of traffic on edge $e$ is $3$ which occurs when all terminals on the left hand side want to communicate with some terminal on the right hand side.}
	\label{fig:tree-reservation}
\end{figure}
In the following we assume that capacities are chosen accordingly in any tree solution.

Although many different (groups of) researchers are working on the \vpn problem, there is no instance known where a tree solution of minimum cost is not simultaneously an optimal virtual private network.  It is widely believed that the following conjecture holds.

\begin{conjecture}[The VPN Tree Routing Conjecture]
\label{con:0}
For each \vpn instance $(G,c,W,b)$ there exists an optimal virtual private network which is a tree solution.
\end{conjecture}

The only progress in this direction is due to Hurkens, Keijsper, and Stougie \cite{HKS05} who prove that Conjecture~\ref{con:0} holds on ring networks and for some other special cases of the problem. Their results are based on a linear programming (LP) formulation of a relaxation of the \vpn problem.  In this relaxation several paths may be chosen between each pair of terminals but the fraction of traffic along each of these paths must be fixed, i.e., it may not depend on the actual set of traffic demands.  Hurkens, Keijsper, and Stougie construct solutions to the dual linear program whose cost equals the cost of particular tree solutions.  The details of their proof are somewhat involved.

In the following we present a simpler proof of this result that is based on a new and stronger conjecture which might be of independent interest.  In Section~\ref{sec:preliminaries} we argue that we can restrict to instances with unit communication bounds.  Our new conjecture is presented in Section~\ref{sec:PR}.  Finally we give a short proof for the case of ring networks in Section~\ref{sec:ring}.

\section{Preliminaries}
\label{sec:preliminaries} 

As already observed in \cite{HKS05}, when proving Conjecture~\ref{con:0}, we may assume that $b(i)=1$ for each terminal $i\in W$.  For the sake of self-containedness, we motivate this assumption in the following.  For more details we refer to \cite{HKS05}.

Consider a terminal $i\in W$ with $b(i)\geq 2$; see Figure~\ref{fig:splitting}~(i) for an example.  
\begin{figure}[tb]
	\centering
		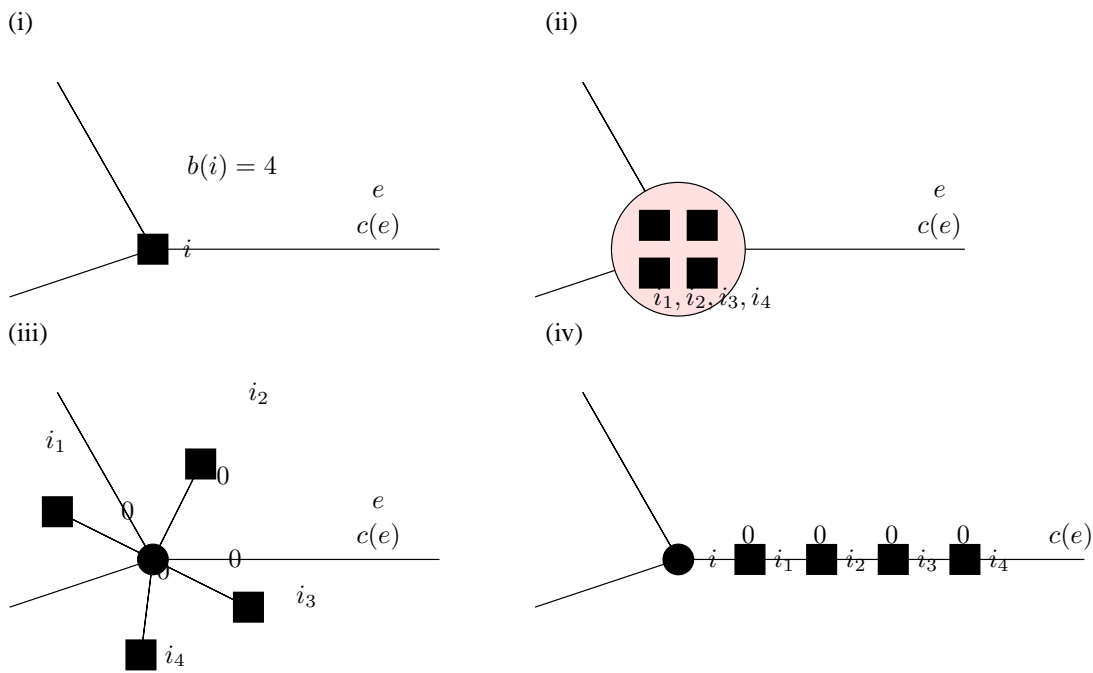
	\caption{Illustration of the reduction described in Section~\ref{sec:preliminaries}.  As an example, we consider a terminal $i\in W$ with $b(i)=4$ and three incident edges; see part~(i) of the figure.  Part~(ii) depicts the relaxed instance where the terminal at node~$i$ is split into $4$ sub-terminals with unit communication bounds.  In order to get an instance with at most one terminal at every node, we can introduce $4$ new terminal nodes, one for each sub-terminal, and connect them to node~$i$ by edges of zero cost; see part~(iii).  Alternatively, in order to stay closer to the original topology of the network, one can subdivide an edge~$e$ incident to node~$i$ and place the terminals there with zero connection cost to~$i$; see part~(iv).}
	\label{fig:splitting}
\end{figure}
We construct a new \vpn instance by replacing terminal $i$ with $b(i)$ ``\emph{sub-terminals}'' $i_1,\dots,i_{b(i)}$ with $b(i_j)=1$ for $j=1,\dots,b(i)$ that are all collocated at the same node $i\in V$; see Figure~\ref{fig:splitting}~(ii).  Strictly speaking, the collocation of terminals at a node is a slight extension of the original definition of the \vpn problem.  We show how to deal with this below.

Obviously, the new instance is a relaxation of the original instance.  The additional degree of freedom of the new instance is that traffic from terminal $i$ to some other terminal no longer needs to be routed along one fixed path but may be split into $b(i)$ packets of equal size that can be routed along different (but also fixed) paths.  

In order to show that there exists an optimal virtual private network for the original instance which is a tree, it is obviously sufficient to find such an optimal solution to the new instance.

We finally argue that the new instance is itself equivalent to an \vpn instance (in the strict sense) with all $b(i)$'s equal to $1$.  Instead of having $b(i)$ sub-terminals located at node~$i$, we add new terminal nodes $i_1,\dots,i_{b(i)}$ to the network and connect them to node $i$ with edges of cost~$0$; see Figure~\ref{fig:splitting}~(iii).  It is easy to observe that every feasible (tree) solution to the instance in Figure~\ref{fig:splitting}~(ii) naturally corresponds to a feasible (tree) solution to the instance in Figure~\ref{fig:splitting}~(iii) of the same cost and vice versa. 

The only problem with the described approach is that the resulting network no longer has the same topology as the network of the original instance.  For example, if we start with a ring network, the resulting network in Figure~\ref{fig:splitting}~(iii) is not a ring.  This problem can be resolved by using the alternative construction illustrated in Figure~\ref{fig:splitting}~(iv).  Here we subdivide an arbitrary edge~$e$ incident to node~$i$ into~$b(i)+1$ parts by introducing~$b(i)$ new terminal nodes $i_1,\dots,i_{b(i)}$. The first $b(i)$ edges connecting the new terminal nodes to node~$i$ have cost~$0$.  To the last edge we assign cost~$c(e)$.  Again, it is easy to observe that every feasible solution to the instance in Figure~\ref{fig:splitting}~(ii) naturally corresponds to a feasible solution to the instance in Figure~\ref{fig:splitting}~(iv) of the same cost and vice versa.  

In order to argue that also tree solutions correspond to each other, one has to be a little more careful.  Notice that not every tree solution to the instance in Figure~\ref{fig:splitting}~(iv) induces a tree solution to the instance in Figure~\ref{fig:splitting}~(ii) (e.g., if terminals~$i_1$ and~$i_2$ are leaves of the tree).  But  there always exists an \emph{optimal} tree solution to the instance in Figure~\ref{fig:splitting}~(iv) containing all edges $i_1i_2,\dots,i_{b(i)-1}i_{b(i)}$.  This follows from the fact that these edges have cost~$0$ and an optimal tree solution is a shortest-paths tree for some source node~$j\in V$; see~\cite{GuptaKleinbergKumarRastogiYener01}.

We conclude this section with the following lemma resulting from our considerations above.

\begin{lemma}\label{lem:b=1}
  Consider a class of networks which is closed under the operation of subdividing edges by introducing new nodes of degree~$2$.  Then Conjecture~\ref{con:0} holds for this class of networks if and only if it holds for the restricted class of instances on such networks where~$b(i)=1$ for all~$i\in W$.
\end{lemma}

\section{The Pyramidal Routing Problem}
\label{sec:PR}

Consider again a tree solution to an \vpn instance as depicted in Figure~\ref{fig:tree-reservation} with $k=|W|$ terminals and $b(i)=1$ for all terminals~$i\in W$.  For a fixed terminal~$i\in W$, let~$\paths_i$ denote the set of simple paths $P_{ij}$ in the tree that connect $i$ to the remaining terminals $j\in W\setminus\{i\}$.  Notice that the required capacity of edge $e$ can be written as
\begin{align*}
	u(e)=\min\{n(e,\paths_i),k-n(e,\paths_i)\}
\end{align*}
where $n(e,\paths_i)$ is the number of paths in $\paths_i$ containing edge~$e$, that is,
\begin{align}
  n(e,\paths_i)&:=|\{j\in W\setminus\{i\}\mid e\in P_{ij}\}|\enspace.\label{eq:n}	
\intertext{To simplify notation in the following we let}
  y(e,\paths_i)&:=\min\{n(e,\paths_i),k-n(e,\paths_i)\}\enspace;\label{eq:y}
\end{align}
see also Figure~\ref{fig:concave_cost}.
\begin{figure}[tb]
	\centering
		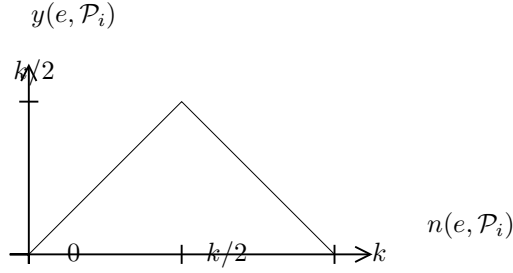
	\caption{The concave cost function of the Pyramidal Routing Problem.}
	\label{fig:concave_cost}
\end{figure}

We now introduce another problem that we call \emph{Pyramidal Routing} (\crp) problem. The problem is defined by $G$, $c$, $W$---as above for the \vpn problem---and a \emph{single source terminal} $i\in W$.  A solution to the Pyramidal Routing Problem consists of a set $\paths_i$ of \emph{simple} $i$-$j$-paths $P_{ij}$, one path for each terminal $j\in W\setminus\{i\}$.  As above we denote the number of paths in $\paths_i$ containing a fixed edge $e\in E$ by $n(e,\paths_i)$; see \eqref{eq:n}.  Moreover, we define $y(e,\paths_i)$ accordingly as in \eqref{eq:y}. The \crp problem is to find a solution $\paths_i$ that minimizes the objective function
\begin{displaymath}
  \sum_{e\in E} c(e)\,y(e,\paths_i)\enspace.
\end{displaymath}

The \crp problem can be seen as an unsplittable flow problem with concave cost functions on the edges (see Figure~\ref{fig:concave_cost}) where one unit of flow is sent from the source~$i$ to all destinations~$j\in W\setminus\{i\}$.  A \emph{tree solution} to \crp is a solution where the chosen paths $P_{ij}$, $j\in W\setminus\{i\}$, form a tree.

\begin{lemma}\label{lem:tree-solutions}
  Consider an \vpn instance $(G,c,W,b)$ with~$b_j=1$ for all~$j\in W$ and a corresponding \crp instance $(G,c,W,i)$ for some terminal $i\in W$.  Then any tree solution to the \vpn instance yields a tree solution of the same cost to the \crp instance and vice versa.
\end{lemma}

\begin{proof}
  Consider any subtree of $G$ containing all terminals. The induced set of $i$-$j$-paths, $j\in W\setminus\{i\}$, is $\paths_i$.   For the \vpn problem, the required capacity of some tree edge $e$ is equal to $y(e,\paths_i)=\min\{n(e,\paths_i),k-n(e,\paths_i)\}$.  Therefore the cost of this tree solution is equal to $\sum_{e\in E}c(e)\,y(e,\paths_i)$ for both problems.  
\end{proof}

We state an immediate corollary of the last lemma.

\begin{corollary}\label{cor:PR-trees}
  A tree solution of minimum cost to a \crp instance $(G,c,W,i)$ yields simultaneously a tree solution of (the same) minimum cost to the \crp instances $(G,c,W,j)$ for all~$j\in W$.  In particular, $y(e,\paths_i)=y(e,\paths_j)$ where $\paths_i$ and $\paths_j$ denote the two sets of paths connecting terminal $i$ and $j$, respectively, with all other terminals.
\end{corollary}

We state the following conjecture for the Pyramidal Routing problem.

\begin{conjecture}[The Pyramidal Routing Conjecture]\label{con:PR}
  For each \crp instance $(G,c,W,i)$ there exists an optimal solution which is a tree solution.
\end{conjecture}

As we show in the remainder of this section, Conjecture~\ref{con:PR} is strongly related to Conjecture~\ref{con:0}.

\begin{theorem}\label{thm:equivalent-conj}
  If Conjecture~\ref{con:PR} holds on some class of networks, which is closed under the operation of subdividing edges by introducing new nodes of degree~$2$, then Conjecture~\ref{con:0} holds on the same class.
\end{theorem}

Before we can prove this theorem we need one further result.  The following lemma and its proof are a slight extension of Theorem 3.2 in~\cite{GuptaKleinbergKumarRastogiYener01}.

\begin{lemma}\label{lem:1}
  Consider an \vpn instance $(G,c,W,b)$ with $b(j)=1$ for all $j\in W$ and some feasible virtual private network given by simple $i$-$j$-paths $P_{ij}$, $i,j\in W$, and capacities $u(e)$, $e\in E$.  There exists a terminal $i\in W$ such that $\sum_{e\in E} c(e)\,u(e)\geq\sum_{e\in E}c(e)\,y(e,\paths_i)$, where $\paths_i=\{P_{ij}\mid j\in W\setminus\{i\}\}$.
\end{lemma}

\begin{proof}
We first consider some fixed edge $e\in E$.  In order to derive a lower bound on $u(e)$, we define a set of traffic demands $D^e$ by setting
\begin{align*}
  d^e_{ij}=
  \begin{cases}
  	\frac{1}{k}\left(\frac{y(e,\paths_i)}{n(e,\paths_i)}+\frac{y(e,\paths_j)}{n(e, \paths_j)}\right) & \text{if~}e\in P_{ij}~,\\
    0 & \text{if~}e\not\in P_{ij}~.
  \end{cases}
\end{align*}

\begin{claim}\label{cla:1}
  $D^e$ is a set of valid traffic demands.
\end{claim}

\begin{proof}
We need to show that $\sum_{j\in W}d^e_{ij}\leq1$ for each $i\in W$.  Remember that $y(e,\paths_i)=\min\{n(e,\paths_i),k-n(e,\paths_i)\}$.  By definition of $D^e$ we thus get
\begin{align*}
  \sum_{j\in W}d^e_{ij}&=\frac{1}{k}\sum_{j\in W: e\in P_{ij}}\left(\frac{y(e, \paths_i)}{n(e, \paths_i)} + \frac{y(e, \paths_j)}{n(e, \paths_j)}\right)\\ 
&\leq \frac{1}{k} \sum_{j\in W: e\in P_{ij}}\left(\frac{k-n(e,\paths_i)}{n(e,\paths_i)}+\frac{n(e,\paths_j)}{n(e,\paths_j)}\right)
=\frac{1}{k}\sum_{j\in W: e\in P_{ij}}\frac{k}{n(e,\paths_i)}=1\enspace.
\end{align*}
The last equation follows since $|\{j\in W\mid e\in P_{ij}\}|=n(e,\paths_i)$.
\end{proof}

\begin{claim}\label{cla:2}
  $u(e)\geq\frac{1}{k}\sum_{i\in W}y(e,\paths_i)$~.
\end{claim}

\begin{proof}
We know from the previous claim that $D^e$ is a valid set of traffic demands. Moreover, by definition, $d^e_{ij}>0$ implies $e\in P_{ij}$. Therefore 
\begin{align*}
  u(e)&\geq\frac{1}{k}\sum_{\{i,j\}\in\binom{W}{2}: e\in P_{ij}}\left(\frac{y(e, \paths_i)}{n(e, \paths_i)} + \frac{y(e, \paths_j)}{n(e, \paths_j)}\right) 
=\frac{1}{k}\sum_{i\in W: n(e,\paths_i)>0} y(e,\paths_i) 
=\frac{1}{k}\sum_{i\in W} y(e,\paths_i)
\end{align*}
since $y(e, \paths_i) = 0$ if $n(e, \paths_i) = 0$.
\end{proof}

\noindent
Since Claim~\ref{cla:2} holds for all $e\in E$, it follows that
\begin{align*}
\sum_{e\in E}c(e)\,u(e)&\geq\sum_{e\in E}c(e)\,\frac{1}{k}\sum_{i\in W}y(e,\paths_i) 
=\frac{1}{k}\sum_{i\in W}\sum_{e\in E}c(e)\,y(e,\paths_i) 
\geq\min_{i\in W}~\sum_{e\in E}c(e)\,y(e,\paths_i)\enspace.
\end{align*}
This concludes the proof of the lemma.
\end{proof}

We can now prove Theorem~\ref{thm:equivalent-conj}.

\begin{proof}[Proof of Theorem~\ref{thm:equivalent-conj}]
  Consider an \vpn instance $(G,c,W,b)$ with~$b_j=1$ for all~$j\in W$; this can be assumed without loss of generality due to Lemma \ref{lem:b=1}. Consider an optimal solution to the instance.  By Lemma~\ref{lem:1} there exists a terminal $i\in W$ such that the optimal solution value of the \crp instance $(G,c,W,i)$ is a lower bound on the optimal solution value of the \vpn instance.  If Conjecture~\ref{con:PR} holds for the \crp instance, there exists an optimal tree solution to the \crp instance which, by Lemma~\ref{lem:tree-solutions}, also yields an optimal virtual private network.
\end{proof}

\section{The Case of Ring Networks}
\label{sec:ring}

In this section we prove Conjecture~\ref{con:PR} for the case that $G$ is an arbitrary ring network, i.e., a cycle. 

\begin{theorem}
\label{theo:PR-ring}
Conjecture~\ref{con:PR} holds true when $G$ is a ring network.
\end{theorem}

\begin{proof}
Consider a \crp instance $(G,c,W,i)$ where $G$ is a cycle.	 We can assume without loss of generality that each node $v\in V$ is a terminal, i.e., $W=V$: If there is a non-terminal node $v$ with two neighbors $x,y\in V$, we can remove~$v$ and replace edges $xv$ and $vy$ by a new edge $xy$ with $c(xy):=c(xv)+c(vy)$.
	
Let $\paths_i$ be an arbitrary optimal solution to this \crp instance.  

\begin{claim}\label{claim1}
Let $e,f\in E$ be a pair of incident edges with~$i\not\in e\cap f$ (i.e., not both edges are incident to terminal $i$).  Then 
  $n(e,\paths_i)=n(f,\paths_i)\pm1$.
\end{claim}

\begin{proof}
  Let $j\in e\cap f$ be the terminal that is incident to both $e$ and $f$.  Since $j\neq i$, exactly one path in $\paths_i$ (namely the $i$-$j$-path) ends in $j$.  All other paths either contain both edges $e$ and $f$ or none of the two.
\end{proof}

\begin{claim}\label{claim2}
Let $e,f\in E$ be an arbitrary pair of incident edges.  Then 
  $y(e,\paths_i)=y(f,\paths_i)\pm1$.
\end{claim}

\begin{proof}
  If~$i\not\in e\cap f$, the claim follows immediately from Claim~\ref{claim1}.  We can therefore assume that $e$ and $f$ are the two edges incident to terminal $i$.  Since there are exactly $k-1$ simple paths in $\paths_i$ starting at $i$, we get $n(e,\paths_i)+n(f,\paths_i)=k-1$.  The result follows by definition of $y(e,\paths_i)$ and $y(f,\paths_i)$; see \eqref{eq:y}.
\end{proof}

\noindent
We distinguish two cases.

\textbf{First case:} $k$ is even.  Since $n(e,\paths_i)+n(f,\paths_i)=k-1$ for the two edges $e$ and $f$ incident to $i$, we can assume that $n(e,\paths_i)\geq\frac{k}{2}$ and $n(f,\paths_i)<\frac{k}{2}$.  Going along the cycle, it follows from Claim~\ref{claim1} that there exists an edge~$e_0$ with $n(e_0,\paths_i)=\frac{k}{2}$.  We number the edges along the cycle consecutively $e_0,e_1,\dots,e_{k-1}$. Using Claim~\ref{claim2} inductively, it follows that
\begin{align*}
  y(e_{\ell},\paths_i)\geq |\tfrac{k}{2}-\ell|\qquad\text{for~$\ell=0,\dots,k-1$;}
\end{align*}
see Figure~\ref{fig:cycle}.
\begin{figure}[tb]
	\centering
		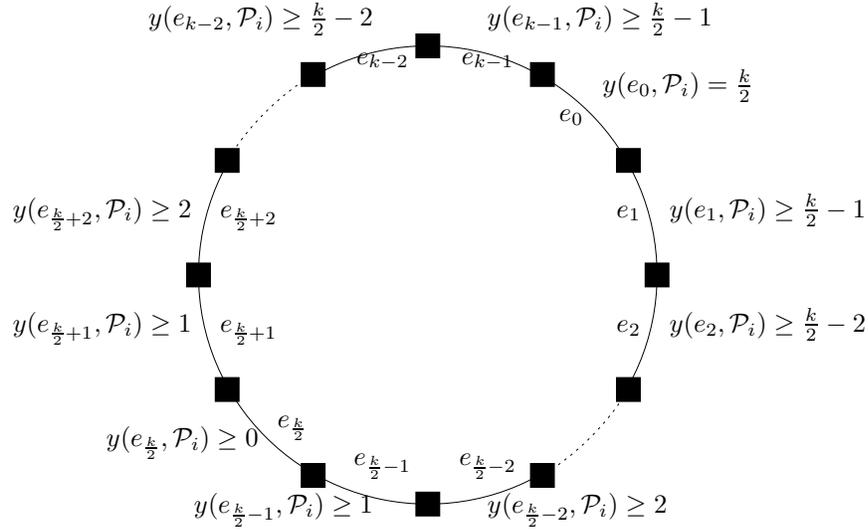
	\caption{Lower bounding the $y$-values along the cycle.}
	\label{fig:cycle}
\end{figure}

Removing edge $e_{\frac{k}{2}}$ from the cycle yields a tree and a corresponding tree solution $\paths_i'$ to the \crp instance.  Moreover, we argue that
\begin{align*}
  y(e_{\ell},\paths_i')=|\tfrac{k}{2}-\ell|\qquad\text{for~$\ell=0,\dots,k-1$.}
\end{align*}
By Corollary~\ref{cor:PR-trees} we can assume without loss of generality that terminal~$i$ is incident to edge~$e_0$. But in this case $n(e_{\ell},\paths_i')=|\tfrac{k}{2}-\ell|$ and thus $y(e_{\ell},\paths_i')=|\tfrac{k}{2}-\ell|$ for~$\ell=0,\dots,k-1$.

As a result, $\sum_{e\in E}c(e)y(e,\paths_i)\geq\sum_{e\in E}c(e)y(e,\paths_i')$.  Since $\paths_i$ is an optimal solution, the tree solution $\paths_i'$ is optimal as well.

\textbf{Second case:} $k$ is odd.  In this case it follows from Claims~\ref{claim1} and \ref{claim2} that there exist two incident edges that we call~$e_{\frac{1}{2}}$ and $e_{k-\frac{1}{2}}$ with
\begin{align*}
y\left(e_{\frac{1}{2}},\paths_i\right)=y\left(e_{k-\frac{1}{2}},\paths_i\right)=\frac{k-1}{2}\enspace.	
\end{align*}
We number the edges along the cycle consecutively $e_{\frac{1}{2}},e_{\frac{1}{2}+1},\dots,e_{k-\frac{1}{2}}$. Using Claim~\ref{claim2} inductively, it follows that
\begin{align*}
  y(e_{\ell},\paths_i)\geq|\tfrac{k}{2}-\ell| \qquad\text{for~$\ell=\tfrac{1}{2},\dots,k-\tfrac{1}{2}$.}
\end{align*}
The remainder of the proof is identical to the first case.
\end{proof}

Theorem~\ref{thm:equivalent-conj} and Theorem~\ref{theo:PR-ring} imply that the VPN Tree Routing Conjecture is true for ring networks.

\begin{corollary}[\cite{HKS05}]\label{cor:vpn-conj}
  Conjecture~\ref{con:0} holds true if $G$ is a ring network.
\end{corollary}


We conclude the paper with some remarks and observations. First, in the definition of the Pyramidal Routing problem we might relax the constraints that the paths have to be simple and allow \emph{trails}; a trail may visit nodes (but not edges) more than once. It is easy to see that Lemma \ref{lem:tree-solutions}, Lemma \ref{lem:1}, and Theorem \ref{thm:equivalent-conj} still hold.
(Observe that on a ring a path is a trail and vice versa.)

Then, in their paper \cite{HKS05}, Hurkens, Keijsper, and Stougie prove a result that is slightly stronger than Corollary~\ref{cor:vpn-conj}.  Consider the relaxation of the \vpn problem where traffic between each pair of terminals may be routed along several different paths but the fraction of traffic along each of these paths must be fixed.  Using the notation from \cite{HKS05}, we refer to the relaxed problem as the Multipath Routing (\mr) problem.  Hurkens et al.\ show that, for the case of ring networks, even \mr has always an optimal solution which is a tree solution.  We observe that this stronger result holds in general if Conjecture~\ref{con:0} is true.

\begin{observation}
	Consider a class of networks which is closed under the operation of subdividing edges by introducing new nodes of degree~$2$.  If Conjecture~\ref{con:0} holds for this class of networks, then also the \mr problem has optimal solutions which are tree solutions for this class of networks. 
\end{observation}

\begin{proof}[Sketch of proof.]
  As the $b(i)$ are integers and there is only a finite number of sets of
paths one finds that the fractionalities in the solution to the \mr
problem can be restricted to be rationals with bounded denominators for
every particular instance of the \mr problem. Subdividing the terminals
as in Section 2 by a number of nodes depending on a bound on these
denominators one constructs an \vpn instance whose optimal (tree)
solution yields an optimal (tree) solution to the original instance of
the \mr problem.
\end{proof}

\paragraph{Acknowledgements:} The authors wish to thank Yuri Faenza for an useful discussion concerning Lemma \ref{lem:b=1}. The third author has been supported by the Italian Research Project Cofin 2006  ``Modelli ed algoritmi per l'ottimizzazione robusta delle reti''. The last author has been supported by the DFG Focus Program 1126, ``Algorithmic Aspects of Large and Complex Networks'', grant SK 58/5-3.


\begin{thebibliography}{1}


\bibitem{ERL}
T.~Erlebach and M.~R\"{u}egg.
\newblock
Optimal Bandwidth Reservation in Hose-Model VPNs with Multi-Path Routing.
\newblock In {\em Proceedings of the Twenty-Third Annual Joint Conference of the IEEE Computer and Communications Societies INFOCOM}, 4:2275-2282, 2004.

\bibitem{GuptaKleinbergKumarRastogiYener01}
A.~Gupta, J.~Kleinberg, A.~Kumar, R.~Rastogi, and B.~Yener.
\newblock Provisioning a virtual private network: A network design problem for
  multicommodity flow.
\newblock In {\em Proceedings of the 33rd ACM Symposium on Theory of
  Computing}, 389--398, ACM Press, 2001.

\bibitem{HKS05}
C.~A.~J. Hurkens, J.~C.~M. Keijsper, and L.~Stougie.
\newblock Virtual private network design: A proof of the tree routing
  conjecture on ring networks.
\newblock {\em SIAM Journal on Discrete Mathematics}, 21:482--503, 2007.

\bibitem{ILO02}
G.~Italiano, S.~Leonardi, and G.~Oriolo.
\newblock Design of trees in the hose model: the balanced case.
\newblock {\em Operation Research Letters,}, 34(6): 601--606, 2006.



\end{thebibliography}

%
%
%
%

\end{document}